\documentclass[final]{dmtcs-episciences}

\usepackage[utf8]{inputenc}
\usepackage{subfigure}
\usepackage{amssymb}
\usepackage{amsfonts}
\usepackage{latexsym}
\usepackage{amsmath}

\newtheorem{theorem}{Theorem}
\newtheorem{remark}{Remark}
\newtheorem{proposition}{Proposition}
\newtheorem{definition}{Definition}
\newtheorem{lema}{Lemma}
\usepackage[round]{natbib}

\author{Ruy Fabila-Monroy\affiliationmark{1}\thanks{was partially supported by CONACyT of Mexico grant 253261.}
  \and Carlos Hidalgo-Toscano\affiliationmark{2}\thanks{was partially supported by CONACyT of Mexico grant 253261}
  \and Jes\'us Lea\~nos \affiliationmark{3} 
  \and Mario Lomel\'{\i}-Haro \affiliationmark{4}\thanks{Gratefully acknowledges CONACyT for doctorate scholarship 163435.}}
\title[The Chromatic Number of the Disjointness Graph of the Double Chain]{The Chromatic Number of the Disjointness Graph of the Double Chain}

\affiliation{
  Departamento de Matem\'aticas, Centro de Investigaci\'on y de Estudios Avanzados del Instituto Polit\'ecnico Nacional. M\'exico.\\
  Centro de Investigaci\'on e Innovaci\'on en Tecnolog\'ias de la Informaci\'on y Comunicaci\'on, M\'exico\\
  Unidad Acad\'emica de Matem\'aticas, Universidad Aut\'onoma de Zacatecas, M\'exico\\
  Instituto de F\'isica, Universidad Aut\'onoma de San Luis Potos\'{\i}, M\'exico
}
\received{2018-5-9}
\revised{2019-2-20, 2019-9-3}
\accepted{2020-2-13}
\begin{document}
\publicationdetails{22}{2020}{1}{11}{4490}
\maketitle
\begin{abstract}
  Let $P$ be a set of $n\geq 4$ points in general position in the plane.
Consider all the closed straight line segments with  both endpoints in $P$. 
Suppose that these segments are colored with the rule that disjoint segments
receive different colors. 
In this paper we show that
if $P$ is the point configuration known as the double chain, with $k$ points
in the upper convex chain and $l \ge k$ points in the lower convex chain, then 
$k+l- \left\lfloor \sqrt{2l+\frac{1}{4}} - \frac{1}{2}\right\rfloor$ colors
are needed and that this number is sufficient.
\end{abstract}

\section{Introduction}\label{sec:introduction}

Throughout this paper,  $P$ is a set of 
$n\geq 4$ points in general position in the plane. The {\em edge disjointness graph},  $D(P)$, of $P$ is the graph whose vertices are all the closed straight 
line segments with endpoints in $P$; two of which are adjacent in $D(P)$ if and only if they are disjoint. 
The edge disjointness graph and  other similar graphs were introduced by 
~\cite{gaby}, 
as geometric analogs of the well known  Kneser graphs. Let $m$ and $k$ be positive integers with $k\leq  m/2$.  We recall that the \emph{Kneser graph} 
$KG(m; k)$ is the graph whose vertices are all the $k$--subsets of $\{1,2,\ldots ,m\}$; two of which are adjacent if and only if they correspond to disjoint 
$k$-subsets.

The \emph{chromatic number} of a graph $G$ is the minimum number of colors needed to color its vertices
so that adjacent vertices receive different colors; it is denoted by $\chi (G)$. 
\cite{kneser} posed the problem of finding the chromatic number 
of the Kneser graph. He conjectured that \[\chi(KG(n; k))=n-2k+2\] for $n \ge 2k-1$.
The upper bound can be shown with simple
combinatorial arguments. The lower bound was proved by 
\cite{lovasz} using
tools from algebraic topology (specifically the Borsuk-Ulam theorem). This is one
of the earliest applications of Algebraic Topology to combinatorial problems. 
For a nice account of this connection see the book of~\cite{using}. 

Recently, 
\cite{Pach-Tomon} have proved that if $G$ is the disjointness graph of a family of grounded $x$-monotone curves such that $\omega(G)=k$, then $\chi(G)\leq \binom{k+1}{2}$, 
where $\omega(G)$ denotes the clique number of $G$. We remark that 
the family of grounded $x$-monotone curves play the role of our closed straight line segments.

Clearly, the chromatic number is a well studied parameter of the Kneser graph
and its relatives. A general upper bound of 
\[\chi(D(P)) \leq \min\left\{n-2, n+\frac{1}{2}-\frac{\lfloor \log \log n\rfloor}{2}\right\}\]
was proved 
by \cite{gaby}.
They obtained it as follows. Let $C_n$ 
be a set of $n$ points in convex position in the plane. Let
\[ f(n) := \chi (D(C_n)). \]
They  showed that $f(n) \le n-\frac{\lfloor \log_ 2 n\rfloor}{2}$. 
\cite{happyend} proved that $P$ has a subset of  at least 
$m=\lfloor \log_2 (n)/2 \rfloor$ points in convex position.
The segments with endpoints in this subset are colored using $f(m)$ colors; the remaining segments are colored by deleting 
the remaining points one by one and in the process coloring all the segments with this point as an endpoint with the same new color.

The exact value of $f(n)$ has been computed. It is now known that 

  \begin{equation}\label{eq:f(n)}
f(n)= n - \left\lfloor \sqrt{2n+\frac{1}{4}} - \frac{1}{2}\right\rfloor.
\end{equation}

Indeed, 
\cite{ruy} showed 
that  the expression on the right hand side of Eq. (\ref{eq:f(n)}) is a lower bound for $f(n)$; and 
\cite{jonsson} established Eq. (\ref{eq:f(n)}) by proving that  such an expression is also an upper bound for $f(n)$. 
Repeating the above arguments, we have that \[ \chi(D(P)) \leq n - \left\lfloor \sqrt{\log n+\frac{1}{4}} - \frac{1}{2}\right\rfloor. \]

As far as we know $\{C_n \}_{n=1}^\infty$ is the only infinite family of point configurations\footnote{with different order types, that is.} for which
the exact value of the chromatic number of their disjointness graph
has been computed. In this paper we compute the chromatic number of the disjointness graph of another infinite family of
point configurations, called the double chain.

We now define this family.
A \emph{$k$-cup} is a set of $k$ points in convex position in the plane such that its convex hull 
is bounded from above by an edge. Similarly, an \emph{$l$-cap} is a set of $l$ points in convex position
whose convex hull is bounded from below by an edge.

\begin{definition}\label{def1}
For $k\leq l$, a $(k,l)$--{\em double--chain} is the disjoint union of two point sets 
$U$ and $L$ such that:

\begin{itemize}
\item $U$ is a $k$-cup and $L$ is an $l$-cap; 
\item every point of $L$ is below every straight line determined by two points of $U$; and 
\item every point of $U$ is above every straight line determined by two points of $L$; 
\end{itemize}
\end{definition}

In Figure~\ref{fig:double_chain} we illustrate a $(5,7)$--double--chain and some of its edges. 
Note that Figure~\ref{fig:double_chain} suggests a natural way to construct a $(k,l)$--double--chain 
for any pair  $(k,l)$ of admissible integers. Moreover, it is a routine exercise to show that any two 
$(k,l)$--double--chains are the same (up to order type isomorphism). In view of this,  we 
shall use $C_{k,l}$ to denote any $(k,l)$--double--chain, and we often refer to it simply as
 the {\em double chain}.  Each of the geometric properties of $C_{k,l}$ in next remark follows easily from its definition, 
 and they will be often used, without explicit mention, in our arguments.   

\begin{remark}
Let $U$ and $L$ be the $k$--cup and the $l$--cap of $C_{k,l}$, respectively. Then the 
following holds:
\begin{itemize}

\item If $U'$ and $L'$ are proper subsets of $U$ and $L$, respectively, then the set of 
points that results from $C_{k,l}$ by deleting the points in $U'\cup L'$ remains a double 
chain.
\item Any straight line segment in the frontier of the convex hull of $U$ (respectively, $L$) does not 
cross any other straight line segment joining two points of $L$ (respectively, $U$)
See Figure \ref{fig:double_chain}.
\item Let $g$ be a straight line segment with an endpoint in $U$ and the other one in $L$, 
and let $f$ be a straight line segment joining two points of $X\in \{U,L\}$. If $g$ and 
$f$ intersect each other, then they do at a common endpoint.   
\end{itemize}
\end{remark}

 The double chain was first introduced by 
 ~\cite{flipping}
 as an example
of a set of $n$ points (in general position) whose flip graph of triangulations has diameter $\Theta(n^2)$.
Since then the double chain has been used as an extremal example in various problems
on point sets, see for example~\cite{geometricGraphs, flipDistance, pseudotriangulations, pathGD, universalSets, geometricGraphs2, lowerBoundGeometricGraphs}.

In this paper we show (Theorem~\ref{thm:main}) that for $l \ge 3$ 
\[\chi(D(C_{k,l})) = k + f\left(l \right).\]

Note that for $n$ even and $k=l=n/2$, $C_{\frac{n}{2},\frac{n}{2}}$ is a set of $n$ points
for which 
\[\chi \left (D \left ( C_{\frac{n}{2},\frac{n}{2}} \right ) \right) = n-\left\lfloor \sqrt{n+\frac{1}{4}} - \frac{1}{2}\right\rfloor \ge f(n)+c\sqrt{n}, \]
for some positive constant $c$. So, to color the disjointness graph of 
$C_{\frac{n}{2},\frac{n}{2}}$, more colors are needed than to color the disjointness graph of 
$C_n$. We conjecture that for every $n \ge 3$, and for every set $P$ of $n$ points 
\[\chi(D(P)) \ge f(n).\]

\begin{figure}[ht]

\centering
\includegraphics[width = 0.6\textwidth]{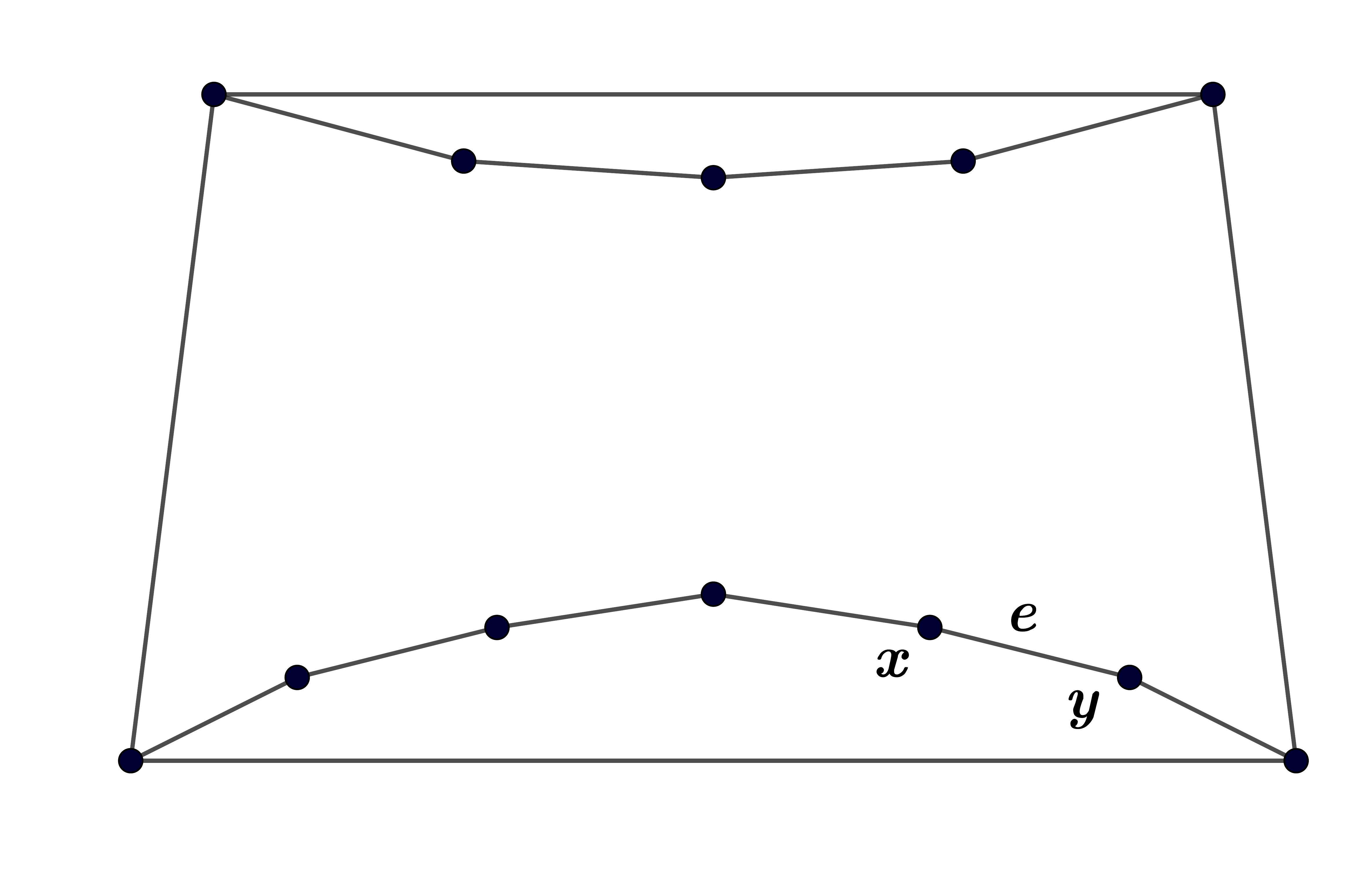}
\caption{\small This is a drawing of $C_{5,7}$ and some of its edges. The edge $e=xy$ is in the convex hull of $L$ and  is not crossed by 
any of remaining edges of $C_{5,7}$. Thus any edge receiving the same color as $e$ in any proper coloring of $D(C_{5,7})$ 
must be incident with exactly one of $x$ of $y$.}
\label{fig:double_chain}
\end{figure}

\section{Preliminary Results and Definitions}\label{sec:preliminaries}

Before proceeding we present some results and definitions. 
A \emph{geometric graph} is a graph whose vertices are points in the plane,
and whose edges are straight line segments joining these points.
For exposition purposes, we abuse notation and use $P$ to refer to
the complete geometric graph with vertex set equal to $P$. 
Thus, $\chi(D(P))$ is the minimum number of colors in an edge-coloring of $P$ in which 
any two edges belonging to the same chromatic class cross or are incident.

Let $c$ be a proper vertex coloring\footnote{a coloring in which pairs of adjacent
vertices receive different colors.} of $D(P)$ and let $S$ be a chromatic class of $D(P)$ 
in this coloring. We say that $S$ is a \emph{star} if all of its edges share a common 
vertex, which we call an \emph{apex}.
If $S$ is not a star then it is a \emph{thrackle}. See Figure~\ref{fig:clasescr}.

\begin{figure}[ht]
\centering
\includegraphics[width = \textwidth]{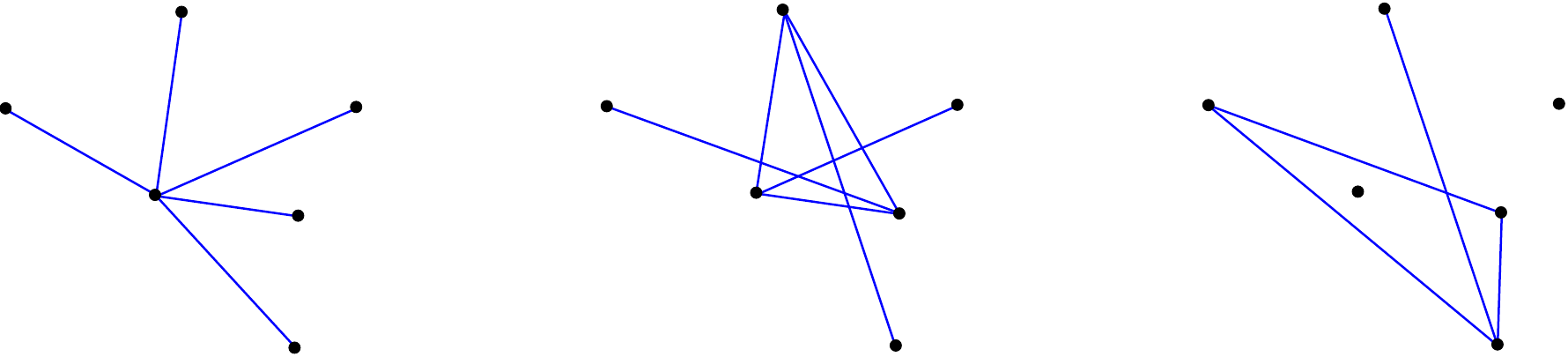}
\caption{\small A star and two distinct thrackles of the same set of $6$ points.}
\label{fig:clasescr}
\end{figure}

\begin{proposition}\label{rstars}
Let $c$ be an optimal coloring of $D(P)$ and let $S_1,\ldots ,S_r$ 
be different stars of $c$ with apices $v_1,\dots, v_r$, respectively. Then  
\[\chi(D(P\setminus\{v_1,\ldots,v_r\})) = \chi(D(P))-r.\]
\end{proposition}
\begin{proof} 
Suppose that there exists a coloring $\chi(D(P\setminus\{v_1,\ldots,v_r\}))$ with less than $\chi(D(P))-r$ colors. Extend this coloring to a coloring of $D(P)$ by using a new different color 
for each $S_i$. This produces a coloring of $D(P)$ with less than $\chi(D(P))$ colors. 
\end{proof}

Let 
  \begin{equation}\label{g(n)}
g(n):= \max\left\{i: i\in \mathbb{Z^+} ,\binom{i}{2} \leq n\right\}.
\end{equation}
\cite{jonsson}  
observed, in the remark following Theorem 1.1, that 
\[f(n)=n-g(n)+1.\]
This implies the following result. 

\begin{proposition}\label{fcreciente}
\begin{equation}\nonumber
f(n+1) = 
\left\lbrace
\begin{tabular}{l l}
 $f(n)$   & if $n = \binom{i}{2}-1$ for some positive integer $i$ and\\
 $f(n)+1$ & otherwise.
\end{tabular}
\right.
\end{equation}

\noindent Therefore, $f(n+k)- f(n)\leq k$, for every nonnegative integer $k$.
\end{proposition}

\begin{proposition}\label{classes}
In every optimal coloring of $D(C_n)$ there is at most one chromatic class consisting of
a single edge of $P$.
\end{proposition}
\begin{proof} 
Suppose for a contradiction that for some $n$ there exists an optimal coloring $c$ of $D(C_n)$ with two chromatic 
classes, $S_1$ and $S_2$, consisting of a single edge. Furthermore, suppose that $n$ is the minimum such integer.
The minimality of $n$ and Proposition \ref{rstars} imply that $S_1$ and $S_2$ are the only stars of $c$.

Let $T_1,\dots, T_k$ be the chromatic classes of $c$ different from $S_1$ and $S_2$. Note that
these are thrackles. 
\cite{ruy} showed that $T_1\cup \cdots \cup T_k$ consists of at most
$kn - \binom{k}{2}$ edges of $C_n$. Therefore, $\binom{n}{2} \le kn - 
\binom{k}{2} +2$.
This implies that $(n-k)^2 \le n+k+4$. Since $k=f(n)-2=n-g(n)-1$, we have that
$(g(n)+1)^2 \le 2n-(g(n)+1)+4.$ Rearranging terms in the previous inequality we
arrive at $\binom{g(n)+1}{2} \le n-g(n)+1$. By the definition of $g(n)$, $\binom{g(n)+1}{2}>n$.
Therefore, $g(n) < 1$ --a contradiction.
\end{proof} 

\section{The Chromatic Number of $D(C_{k,l})$}

It is relatively easy to find an optimal coloring of $D(C_{k,l})$. 

\begin{lema}
 For all positive integers $k \le l$,
 \[\chi(D(C_{k,l})) \le k+f(l).\]
\end{lema}

\begin{proof}
Color the edges of $L$ of $C_{k,l}$ with $f(l)$ colors.
For each of the $k$ vertices in $U$, color the edges incident to them,
that have not been colored yet, with a new color. This yields a proper coloring
of $D(C_{k,l})$ with $k+f(l)$ colors.
\end{proof}

The following lemma is needed to prove the lower bound on $\chi(D(C_{k,l}))$.

\begin{lema}\label{inductionstep} 
If $l \geq 3$, then $\chi(D(C_{1,l})) \ge 1+ f(l).$
\end{lema}
\begin{proof}

From Eq. (\ref{eq:f(n)}) we know that $f(3)=1$. Now we shall show that $\chi(D(C_{1,3}))=1+f(3)=2$. The proper coloring of $D(C_{1,3})$ given in Figure \ref{fig:C13=2} shows that $\chi(D(C_{1,3}))\leq 2$. On the other hand, since the straight line segments $yx_2$ and $x_1x_3$ are disjoint,  then they cannot receive the same color in any proper coloring of $D(C_{1,3})$. 
This implies that $\chi(D(C_{1,3}))\geq 2$, as required. 

\begin{figure}[ht]

\centering
\includegraphics[width = 0.7\textwidth]{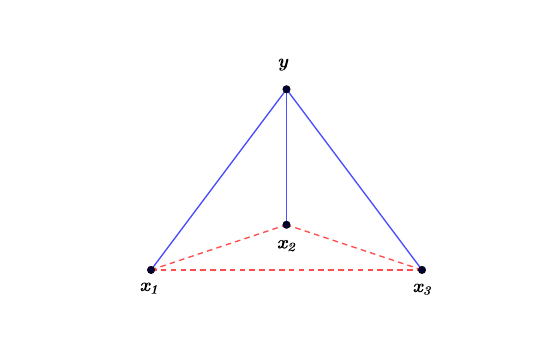}
\caption{\small A proper coloring of $D(C_{1,3})$.}
\label{fig:C13=2}
\end{figure}
  
Assume that $l\ge 4$ and that the result holds for smaller
values of $l$. 
Let $c$ be an optimal coloring of $D(C_{1,l})$. We may assume
that $c$, when restricted to $L$
uses $f(l)$ colors, as otherwise we are done. 

Suppose that $c$ has a star with apex $v$. Then by Proposition~\ref{rstars},
the graph $D(C_{1,l} \setminus \{ v \})$ can be properly colored with one color less.
If $v$ is the single point in $U$, then $c$ uses at least $f(l)+1$ colors.
If $v$ is in $L$, then by induction, $c$ uses at least $\chi(D(C_{1,l-1}))+1=f(l-1)+2$ 
colors. By Proposition~\ref{fcreciente} this is at least $f(l)+1$. Then we can assume 
that all chromatic classes of $c$ are thrackles.

We claim that if all the edges incident to the single vertex $u$ in $U$ are in the same 
chromatic class $H$, then $H$ is a star with apex $u$. Indeed, let $h_1, h_2,\ldots ,h_l$ 
be the edges incident with $u$, and let $w_1,w_2,\ldots ,w_l\in L$ be their respective 
endpoints. Then $\{h_1,h_2,\ldots ,h_l\}\subseteq H$ and $L=\{w_1,w_2,\ldots ,w_l\}$. Now, 
suppose by way of contradiction that there is an edge $w_iw_j$ belonging to $H$. Since 
$l\geq 3$, then there exists a point $w_k\in L\setminus \{w_i,w_j\}$. The existence of 
such $w_k$ and the fact that $L$ is an $l$-cap imply that $h_k$ is disjoint from $w_iw_j$. 
But this contradicts that  $h_k$ and $w_iw_j$ are in the chromatic class $H$. Thus we 
may assume that there are two edges incident to $u$ with different color.  

Let $e_1$ and $e_2$ be two edges incident to $u$ of different colors. Suppose that $e_1$ 
is colored {\em red} and $e_2$ is colored {\em blue}. Let $v_1$ and $v_2$ be their 
respective endpoints in $L$. Since the {\em red} and {\em blue} edges are not stars, 
there exist edges $f_1$ and $f_2$, both with endpoints in $L$, of colors {\em red} and 
{\em blue}, respectively. Note also that all the {\em red} edges of $L$ must be incident 
to $v_1$ and that all the {\em blue} edges of $L$ must be incident to $v_2$. Since the 
{\em red} and the {\em blue} edges are not stars, then there exist other edges incident 
to $u$ of colors {\em red} and {\em blue}. Let $g_1$ and $g_2$ be such edges, and 
suppose that  $g_1$ is {\em red} and that $g_2$ is {\em blue}.

We claim that $f_1$ and $f_2$ are the only {\em red} and {\em blue} edges in $L$. 
Seeking a contradiction, suppose that there exists a {\em red} edge $f'_1\neq f_1$ with 
endpoints in $L$. From previous paragraph we know that both $f_1$ and $f'_1$ are  
incident with $v_1$. Let $v$ and $v'$ be the other  endpoints of $f_1$ and $f'_1$,  
respectively. Then $v\neq v'$, and as a consequence, there is a $w\in \{v,v'\}$ such 
that $w$ is not in $g_1$. This implies that the element of $\{f_1,f'_1\}$ that is 
incident with $w$ is disjoint from $g_1$. This last statement contradicts the assumption 
that  $f_1, f'_1,$ and $g_1$ are all {\em red}. A totally analogous argument shows 
that $f_2$ is the only {\em blue} edge in $L$. Therefore, $c$ when restricted to $L$ is 
an optimal coloring of $C_{l}$ in which two chromatic classes consist of a single edge. 
The last conclusion contradicts Proposition~\ref{classes}, yielding that the restriction 
of $c$ to $L$ is not optimal. This contradicts our earlier supposition that $c$, when 
restricted to $L$ uses $f(l)$ colors.

\end{proof}

\begin{lema} \label{lem:kl}
If $l \geq 3$, then $\chi(D(C_{k,l})) \ge k+f(l).$
\end{lema}
\begin{proof}
Suppose for a contradiction that there exist $k$ and $l$ such that there exists an 
optimal coloring $c$ of $D(C_{k,l})$ with less than $k+f(l)$ colors. Furthermore suppose 
that $k$ and $l$ are such that $k+l$ is minimum.  It can be checked by hand that the 
theorem holds for $k \le l \le 3$, and by Lemma~\ref{inductionstep} it holds for $k=1$. 
Therefore, $k \ge 2$ and $l \ge 4$.
 
Suppose that $c$ has a star with apex $v$. By Proposition~\ref{rstars}, 
$D(C_{k,l}\setminus \{v\})$ can be colored with less than $k+f(l)-1$ colors. If $v$ is in 
$U$ then we have $C_{k,l}\setminus \{v \}=C_{k-1,l}$ and $D(C_{k-1,l})$ can be colored 
with less than $(k-1)+f(l)$ colors; this contradicts the minimality of $k+l$. If $k = l$, 
we can assume without loss of generality that $v$ is in $U$. Thus, we assume that $v$ is 
in $L$ and that $k < l$. Then $C_{k,l}\setminus \{v \}=C_{k,l-1}$ and, by Proposition 
\ref{rstars}, $D(C_{k,l-1})$ can be colored with less than $k+f(l)-1$ colors. By 
Proposition~\ref{fcreciente}, we know that  $k+f(l)-1 \leq k+f(l-1)$; this contradicts 
the minimality of $k+l$. Thus we can assume that all the chromatic classes of $c$ are 
thrackles.
 
Note that there are exactly four edges $e_1$, $e_2$, $e_3$ and $e_4$ in the convex hull 
of  $C_{k,l}$, and let $v_1,v_2,v_3$ and $v_4$ be the set of endpoints of $e_1$, $e_2$, 
$e_3$ and $e_4$. Since each $e_i$ does not cross any other edge, then every edge of the 
same color as $e_i$ must be incident to one of the endpoints of $e_i$. Let $\gamma$ be the 
number of different colors received by these four edges in $c$. Note that $\gamma=2,3$ or 
$4$.
 
Suppose that $\gamma=2$. Without loss of generality assume that $e_1$ and $e_2$ are 
{\em blue}; $e_3$ and $e_4$ are {\em red}; $v_3$ is the common endpoint of $e_1$ and 
$e_2$; and that $v_4$ is the common endpoint of $e_3$ and $e_4$. See Figure \ref{fig:2y4} 
(left). We claim that at least one of these two chromatic classes is a star. Suppose 
that the {\em blue} chromatic class is not a star. Then there is a {\em blue} edge $g$ 
which is not incident to $v_3$. As neither $e_1$ nor $e_2$ is crossed by any other edge, 
then such a $g$ must be $v_1v_2$. Since $g$ is {\em blue} and it is the the only edge 
that intersects both $e_3$ and $e_4$ but not at $v_4$, then the {\em red} chromatic class 
is a star with apex $v_4$, a contradiction to the assumption that all the chromatic 
classes of $c$ are thrackles. 

\begin{figure}[ht]

\centering
\includegraphics[width = 0.9\textwidth]{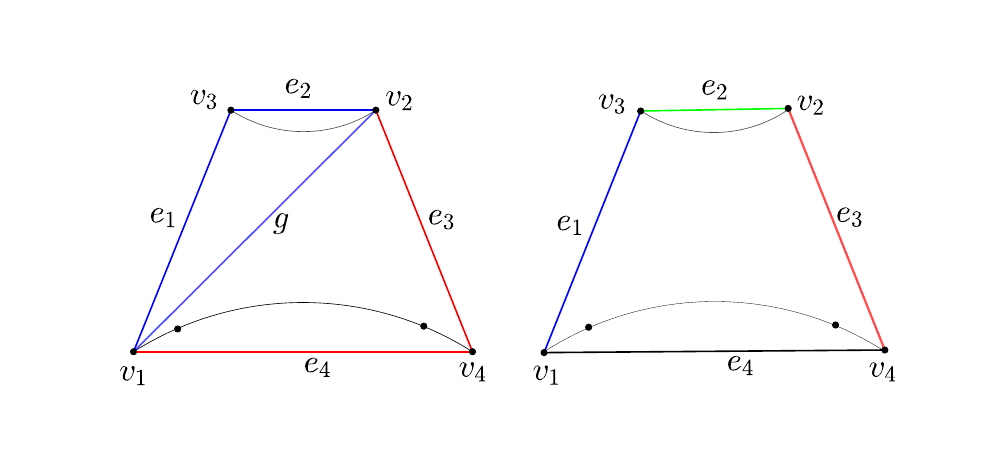}
\caption{\small The cases $\gamma=2$ (left) and  $\gamma=4$ (right) in the proof of 
Lemma~\ref{lem:kl}.}
\label{fig:2y4}
\end{figure} 
  
Suppose that $\gamma=4$. Then there are no edges with the same color as any of the $e_i$ 
in $C_{k,l} \setminus \{v_1,v_2,v_3, v_4\}$. Therefore, $c$ when restricted to the 
subgraph $D(C_{k,l} \setminus \{v_1,v_2,v_3, v_4\})$ uses less than $k+f(l)-4$ colors. 
See Figure \ref{fig:2y4} (right). Note that $C_{k,l} \setminus \{v_1,v_2,v_3, v_4\} 
= C_{k-2,l-2}$; by Proposition~\ref{fcreciente}, $k+f(l)-4$ is at most $(k-2)+f(l-2)$; 
this contradicts the minimality of $k+l$.
 
Finally, suppose that $\gamma = 3$. Then exactly two of the $e_i$ are of the same color; 
moreover these edges share an endpoint. Without loss of generality assume that: these 
edges are $e_1$ and $e_2$; their common endpoint is $v_3$; and that the other endpoints 
of $e_1$ and $e_2$ are $v_1$ and $v_2$, respectively. Assume that $e_1$ and $e_2$ are 
colored {\em blue}. Since all the chromatic classes in $c$ are thrackles then the edge 
$v_1v_2$ must also be colored {\em blue}. Let $S:=U$ if $v_3$ is in $U$ and let $S:=L$ 
if $v_3$ is in $L$. Without loss of generality assume that $v_1$ is not in $S$. Note that 
any other {\em blue} edge must be incident to $v_3$ and its other endpoint is not in $S$. 
Now we recolor {\em blue} all the edges incident with $v_3$ and having the other endpoint 
not in $S$. See Figure \ref{fig:gamma=3}.

\begin{figure}[ht]

\centering
\includegraphics[width = 0.9\textwidth]{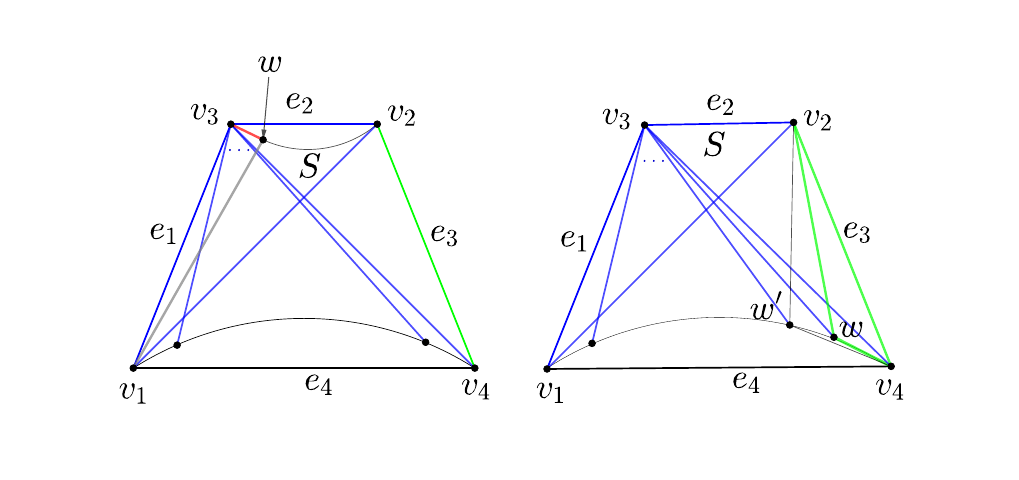}
\caption{\small Here we illustrate the only two (up to symmetry) possibilities for the case $\gamma=3$. 
On the left we have the case in which $|S|\geq 3$ and $S=U$. On the right we have the case in which $|S|=2$ and hence $S=U$.}
\label{fig:gamma=3}
\end{figure}
  
First let us assume that $|S|\geq 3$. We only show the case in which $S=U$. The proof for 
the case $S=L$ is totally analogous. Then $S\setminus \{v_2,v_3\}$ is not empty. Let $w$ 
be the vertex in $S\setminus \{v_2,v_3\}$ which is the closest to $v_3$. See Figure 
\ref{fig:gamma=3} (left). From the definition of $w$ we have that the edge  $v_3w$ does 
not cross any other edge, and in particular $v_3w$ cannot be {\em blue}. Suppose that 
$v_3w$ is {\em red}. If $v_1w$ is also colored {\em red}, then the {\em red} chromatic 
class is a star, a contradiction. Thus $v_1w$ is not red. Since $v_1w$ cannot be colored 
{\em blue}, we assume that it is colored {\em gray}. See Figure \ref{fig:gamma=3} (left). 
Since $v_1w$ is crossed only by {\em blue} edges, then any other {\em gray} edge must be 
incident to $v_1$ or $w$. Also note that every {\em red} edge must be incident to $v_3$ 
or $w$. These observations together imply that $c$ when restricted to 
$C_{k,l} \setminus \{v_1, w, v_3\}$ is a coloring of $D(C_{k,l} \setminus \{v_1, w, v_3\})$ 
with less than $k+f(l)-3$ colors. Then $C_{k,l} \setminus \{v_1, w, v_3\}=C_{k-2,l-1}$. 
By Proposition~\ref{fcreciente}, $k+f(l)-3\le (k-2)+f(l-1)$; this contradicts the 
minimality of $k+l$. 
 
Now suppose that $|S|=2$. Then $S=U=\{v_2,v_3\}$. By symmetry, we may assume that 
$e_1,e_2, e_3$ and $e_4$ are placed as in Figure \ref{fig:gamma=3} (right), and that 
$e_3=v_2v_4$ is {\em green}. Let $w$ be the vertex in $L\setminus\{v_1,v_4\}$ which is 
closest to $v_4$. Then $wv_4$ does not cross any other edge, and any edge crossing $wv_2$ 
is {\em blue}. Also note that $wv_2$ cannot be {\em blue}. If $wv_2$ and $wv_4$ receive the 
same color, different from {\em green}, then  the chromatic class containing them must be a 
star. Similarly, if $wv_2, wv_4$ and $v_2v_4$ receive  distinct colors, then we can proceed 
as in previous paragraph and deduce that $C_{1,l-2}=C_{2,l} \setminus \{v_2, w, v_4\}$ 
is a counterexample that contradicts the minimality of $k+l$. 
  
Thus we may assume that at least one of $wv_2$ or $wv_4$ is {\em green}.  We claim that 
both are {\em green}. Because $v_2v_4$ is not crossed by any edge, then any other {\em 
green} edge must be adjacent to exactly one of $v_2$ or $v_4$. This and the fact that 
the {\em green} chromatic class is not a star, imply that for each $v\in \{v_2, v_4\}$ 
there exists at least one {\em green} edge distinct of $v_2v_4$ which is incident with 
$v$. Let $v_2x$ and $v_4y$ be any couple of such {\em green} edges. Clearly, 
$x,y\in L\setminus \{v_4\}$. Since the {\em green} edges incident with $v_2$ are crossed 
only by {\em blue} edges, then we must have that $x=y$. This and the fact that at least 
one of $wv_2$ or $wv_4$ is {\em green} imply that $w=x=y$. This implies that the 
{\em green} chromatic class consists precisely of $wv_2, wv_4$ and $v_2v_4$.

Let $w'$ be the vertex in $L\setminus\{v_1,w,v_4\}$  which is the closest to $w$.  See 
Figure \ref{fig:gamma=3} (right). Note that $ww'$ does not cross any other edge, and 
that any edge crossing $w'v_2$ is {\em blue}. Also note that none of $w'v_2$ and $w'v_4$ 
can be {\em blue} or {\em green}. Again, if $w'v_2$ and $w'v_4$ receive the same color, 
then the chromatic class containing them must be a star. Thus we assume that $w'v_2$ and 
$w'v_4$ have distinct colors. This implies that the color of at least one of $w'v_2$ or 
$w'v_4$ is different from the color of $ww'$. Let $v\in \{v_2,v_4\}$ such that  
$c(ww')\neq c(w'v)$. Since none of  $ww'$ and $w'v$ can be {\em green}, then the colors of 
$ww', wv,$ and $w'v$ are distinct. From this and the fact that any edge crossing $w'v$ is 
{\em blue} or incident with $w$ it follows that $C_{2,l} \setminus \{v, w, w'\}$ is a 
counterexample that contradicts the minimality of $k+l$. The result follows.
\end{proof}

Summarizing, we have the following result.

\begin{theorem}\label{thm:main}
 For $l \ge 3$, $\chi(D(C_{k,l})) = k+f(l).$
\end{theorem}

\acknowledgements
\label{sec:ack}

We thank two anonymous referees for their valuable comments and improvements to the presentation.

\nocite{*}
\bibliographystyle{abbrvnat}
\bibliography{bibliografia}
\label{sec:biblio}

\end{document}